\RequirePackage{fix-cm}
\documentclass[smallextended]{svjour3}       
\smartqed  
\usepackage{}
\usepackage{graphicx}
\usepackage{multirow,bigstrut}
\usepackage{amsmath,amsfonts,amssymb}
\usepackage{mathrsfs}
\usepackage{color}
\usepackage{indentfirst}
\usepackage{algorithm,algpseudocode}
\usepackage{latexsym}
\usepackage{graphicx}
\usepackage{array,epsf,epsfig}
\usepackage{mathdots}
\usepackage{cite}
\usepackage{multirow}
\usepackage{epstopdf}
\usepackage{mathtools}
\usepackage{tensor}
\usepackage{textcomp}
\usepackage{booktabs}

\newcommand{\mO}{{\mathcal{O}}}

\newcommand{\beq}{\begin{equation}}
\newcommand{\eeq}{\end{equation}}
\newcommand{\bey}{\begin{eqnarray}}
\newcommand{\eey}{\end{eqnarray}}

\begin{document}

\title{Robust fast method for variable-order time-fractional diffusion equations without regularity assumptions \thanks{This work is supported in part by research grants of the Science and Technology Development Fund, Macau SAR (file no. 0122/2020/A3), and University of Macau (file no. MYRG2020-00224-FST).}}

\author{Jia-li Zhang$^1$ \and Zhi-wei Fang$^2$ \and Hai-wei Sun$^3$
}


\institute{1 \at
              Department of Mathematics, University of Macau,
Macao\\
\email{zhangjl2628@163.com}\\
           2 \at
              Department of Mathematics, University of Macau,
Macao\\
\email{fzw913@yeah.net}\\
3 \at
              Corresponding Author. Department of Mathematics, University of Macau,
Macao\\
\email{HSun@um.edu.mo}
}

\date{Received: date / Accepted: date}

\maketitle

\begin{abstract}
In this paper, we develop a robust fast method for mobile-immobile variable-order (VO) time-fractional diffusion equations (tFDEs), superiorly handling the cases of small or vanishing lower bound of the VO function. The valid fast approximation of the VO Caputo fractional derivative is obtained using integration by parts and the exponential-sum-approximation   method. Compared with the general direct method, the proposed algorithm ($RF$-$L1$ formula) reduces the acting memory from $\mathcal{O}(n)$ to $\mathcal{O}(\log^2 n)$ and computational cost from $\mathcal{O}(n^2)$ to $\mathcal{O}(n \log^2 n)$, respectively, where $n$ is the number of time levels. Then $RF$-$L1$ formula is applied to construct the fast finite difference scheme for the VO tFDEs, which sharp decreases the memory requirement and computational complexity. The error estimate for the proposed scheme is studied only under some assumptions of the VO function, coefficients, and the source term, but without any regularity assumption of the true solutions. Numerical experiments are presented to verify the effectiveness of the proposed method.
\keywords {variable-order Caputo fractional derivative \and exponential-sum-approximation method \and fast algorithm \and  convergence
}
\subclass{ 35R11, 65M06, 65M12}
\end{abstract}


\section{Introduction}
Fractional operators have been extensively studied due to their broad applications in both mathematics and physical science. Numerous researchers  revealed that the fractional calculus can better characterize complex phenomena in fields such as the biology, the ecology, the diffusion, and the control system \cite{Benson-2000,   Kilbas-2006,  Liu-2004, Mainardi-2000, Podlubny-1999, Raberto-2002}. In particular, the variable-order (VO) fractional operators are more efficient since many important dynamical problems exhibit the order of the fractional operator varying with time, space, or some other variables; see \cite{Lorenzo-2002, Sun-2019, Sun-2011}. Recently, the VO fractional derivatives have been widely applied to model phenomena in fields of science and engineering; see, for details, \cite{C, Coimbra-2003,  Diazand-2009, Ingman-2004, Jia-2017,  Obembe-2017, Pedro-2008,  Sokolov-2005, Sun-2009, Zhuang-2000}. In this paper, we consider the VO mobile-immobile time-fractional diffusion equations (tFDEs) \cite{ Fu-2019, Zheng-2020, Zheng-2019}

\begin{align}
&\frac {\partial}{\partial t}u(x,t)+\zeta \prescript{C}{0}{\mathcal{D}}^{\alpha(t)}_tu(x,t)=\frac \partial{\partial x}\bigg[p(x)\frac {\partial u(x,t)}{\partial x}\bigg]+f(x,t),\ \ x\in \Omega,\ \ t\in (0,T],\label{E1}\\
&u(x,0)=\varphi(x),\ \ x\in \overline{\Omega},\label{E2}\\
&u(x,t)=0,\ \ x\in \partial\Omega,\ \ t\in (0,T], \label{E3}
\end{align}
where $\Omega=(x_l,x_r)$, $\zeta>0$ is the mobile/immobile capacity coefficient, and $f(x,t)$ is the source term. Moreover, $f(x,t)$, $\varphi(x)$ and $p_*\leq p(x)\leq p^*$ are given sufficiently smooth functions. The VO Caputo fractional derivative is defined by
\cite{Coimbra-2003}
\begin{align}\label{caputo-def}
\prescript{C}{0}{\mathcal{D}}^{\alpha(t)}_tu(x,t):=\frac 1 {\Gamma\big(1-\alpha(t)\big)}\int_0^t\frac{u'(x,\tau)}{(t-\tau)^{\alpha(t)}}\mathrm{d}\tau,
\end{align}
where $0\leq\alpha_*\leq\alpha(t)\leq\alpha^*<1$ is the VO function depending on the time variable $t\in[0,T]$ and $\Gamma(\cdot)$ is the Gamma function. The VO tFDEs \eqref{E1}--\eqref{E3} describe the dynamic mass exchange between mobile and immobile phases and thus improve the modeling of anomalously diffusive transport \cite{Patnaik-2020, Sun-2009}.
Several papers have considered numerical methods for the VO tFDEs; see the references in \cite{Du-2020, Gu-2021, Sun-2012, Zhao-2015}. However, most papers ignored the possible presence of an initial layer in the solution near $t=0$ and presented convergence analyses that makes the unrealistic assumption that the true solution is smooth on the closed domain.
Nevertheless, it is well known that the solutions to tFDEs exhibit initial singularities that may affect the accuracy of the numerical approximations.
In \cite{Zheng-2020, Zheng-2019}, the authors showed that such singularity may not be physical relevant in the diffusion processes and could be eliminated in VO functional models by imposing the integer limit of the VO function at $t=0$. More precisely, the solution has full regularity like its integer-order analogue if the VO function has an integer limit at $t=0$; or exhibits singular behaviors at $t=0$ if the VO function has a non-integer value at $t=0$. Taking into account the initial behavior of the VO tFDEs, a fully discretized finite element approximation to \eqref{E1}--\eqref{E3} is developed and analyzed \cite{Zheng-2020}.

As a result of  the nonlocality of the fractional operators, using $L1$ formula \cite{ Langlands-2005, Liao-2018, Oldham-1974, Sun-2005, Zheng-2020} to discretize the VO Caputo fractional derivative is too expensive in storage and complexity, which requires $\mathcal{O}(n)$ storage and $\mathcal{O}(n^2)$ computational cost, where $n$ is the total number of the time levels.
The constant-order (CO) fractional operators also suffer from such difficulty. Hence many efforts have been made to speed up the evaluation of the CO Caputo fractional derivative \cite{Baffet-2016, Bertaccini-2019,  Fu-2017, Jiang-2017, Ke-2015, Lu-2015,   Lubich-2002, Zeng-2017}.
Nevertheless, the coefficient  matrices of the numerical schemes for the VO fractional problems lose the Toeplitz-like structure and the VO fractional derivative is no longer a convolution operator. Those fast methods for the CO fractional derivative cannot be directly applied to VO cases. Limited fast methods have been presented so far for the VO fractional derivative.
Recently, Fang, Sun and Wang \cite{Fang-2019} developed a fast algorithm for the VO Caputo fractional derivative based on a shifted binary block partition and uniform polynomial approximations.
In \cite{Zhang-2021}, a fast $L1$ formula ($F$-$L1$ formula) is proposed using the exponential-sum-approximation (ESA) technique for the VO Caputo fractional derivative. $F$-$L1$ formula achieves the optimal convergence and significantly reduces the storage requirement to $\mO(\log^2 n)$ and the computational cost to $\mathcal{O}(n\log^2 n)$, respectively.
However, a small value of $\alpha_*$ causes unreliable parameters in the fast algorithm, which leads to inaccurate approximations in the view of scientific computation. Even worse, $F$-$L1$ formula fails to approximate the VO fractional derivative with a vanishing $\alpha_*$ since $\alpha_*=0$ indicates an infinite parameter in the algorithm.

In this paper, to overcome the weakness mentioned above, we develop a valid method to accelerate the approximation of the VO Caputo fractional derivative. At the time level $t_k$, using integration by parts, we transform the singular kernel in \eqref{caputo-def} from $(t_k-\tau)^{-\alpha(t_k)}$ to $(t_k-\tau)^{-1-\alpha(t_k)}$. It makes sense that  $(t_k-\tau)^{-1-\alpha(t_k)}$ can be approached by a linear combination of exponentials based on the ESA technique, which supplies an opportunity to construct the fast algorithm.
Compared with $L1$ formula, it reduces the acting memory from $\mathcal{O}(n)$ to $\mO(\log^2n)$ and computational cost from $\mathcal{O}(n^2)$ to $\mathcal{O}(n\log^2 n)$. Significantly, although $RF$-$L1$ formula and $F$-$L1$ formula are both based on the ESA method, $RF$-$L1$ formula provides a powerful way to fast approximate the VO Caputo fractional derivative, which can attack the problem efficiently even if $\alpha_*=0$ while $F$-$L1$ formula cannot work.
Then $RF$-$L1$ formula is applied to construct a fast finite difference scheme ($RF$-$L1$ scheme) for the VO tFDEs, which sharp decreases the memory requirement and computational complexity. We present the optimal order convergence rate of the proposed scheme, assuming only the smoothness of the coefficients, source term, the VO function and the spatial domain, but not the regularity of the true solution.
The numerical experiments show that $RF$-$L1$ scheme achieves temporal first-order accuracy with less CPU time and memory than the existing methods. Specifically, $RF$-$L1$ scheme performs well for $\alpha_*=0$, $0.05$ while $F$-$L1$ scheme fails to solve the problem.

The structure of the paper is as follows. In Section \ref{Preliminaries}, we refer the regularity and well-posedness of the solution of \eqref{E1}--\eqref{E3} and recall $L1$ and $F$-$L1$ formulas. In Section \ref{fast-formula}, we propose $RF$-$L1$ formula for the VO Caputo fractional derivative based on the ESA technique. In Section \ref{finite-difference-scheme}, $RF$-$L1$ formula is applied to construct a fast difference scheme to solve the VO tFDEs \eqref{E1}--\eqref{E3}. The convergence of the scheme is analysed without the smoothness assumption of the true solution. In Section \ref{numerical-results}, numerical results are reported to demonstrate the efficiency of the proposed scheme. Concluding remarks are given in Section \ref{concluding-remarks}.

\section{Preliminaries}\label{Preliminaries}
\subsection{Well-posedness}
It is well known that the solutions to tFDEs exhibit initial singularities that may affect the accuracy of the numerical approximations. The regularity of all typical solutions of \eqref{E1}--\eqref{E3} is investigated in \cite{Zheng-2020, Zheng-2019}.

In this paper, we use the standard Hilbert space $L_2(\Omega)$ with norm $\|\cdot\|_{L_2}$ and inner product $(\cdot,\cdot)_{L_2}$. For convenience we may drop the subscript $L_2$ in $(\cdot,\cdot)_{L_2}$ and $\|\cdot\|_{L_2}$ as well as the notation $\Omega$ when no confusion occurs. Moreover, $c$ denotes generic positive constants that are independent of $T$ and   any mesh used to solve \eqref{E1}--\eqref{E3} numerically. Note that $c$ may be assumed to be different values at different occurrences. For the remainder of this paper, we make the following assumptions to ensure the essential feature of the solution.\\
\textbf{Condition A:} $\alpha(t)\in C[0,T]$, $0\leq\alpha(t)\leq\alpha^*<1$ on $[0,T]$, $\lim\limits_{t\rightarrow 0^+}\big(\alpha(t)-\alpha(0)\big)\ln t$ exists.\\
\textbf{Condition B:} $p(x)\in C^1(\overline{\Omega})$.

Next, we introduce some notations.
Let $\{(\xi_i, \phi_i ):i=1,2,\ldots\}$ be the eigenvalues and eigenfunctions for the Sturm-Liouville problem
\begin{align*}
&\mathcal{L} \phi_i(x)=\xi_i \phi_i(x),\ \ x\in\Omega,\\
&\phi_i(x)=0,\ \ x\in \partial\Omega,
\end{align*}
where the eigenfunctions are normalised by requiring $\|\phi_i\|=1$. The fractional power operator $\mathcal{L}^\gamma$ for any $\gamma\geq0$ and the fractional Sobolev spaces can be defined by the theory of sectorial operators \cite{A-2003, T-1984}
\begin{align*}
&\mathcal{L}^\gamma v:=\sum\limits_{i=1}^\infty\xi_i^\gamma(v,\phi_i)\phi_i,\ \ v=(v,\phi_i)\phi_i,\\
&\check{H}^\gamma(\Omega)=\{v\in L_2(\Omega): |v|^2_{\check{H}^\gamma}:=(\mathcal{L}^\gamma v,v)=\sum\limits_{i=1}^\infty\xi_i^\gamma(v,\phi_i)^2<\infty\},
\end{align*}
with the norm being defined by $\|v\|_{\check{H}^\gamma}=\big(\|v\|^2+ |v |^2_{\check{H}^\gamma}\big)^{1/2}$. Moreover, $\check{H}^\gamma(\Omega)$ is a subspace of the fractional Sobolev space $H^\gamma(\Omega)$ characterized by \cite{A-2003, T-1984}
\begin{align*}
\check{H}^\gamma(\Omega)=\{v\in H^\gamma(\Omega): \mathcal{L}^s v(x)=0, x\in\partial\Omega, s<\gamma/2\},
\end{align*}
and the seminorms $|v|_{\check{H}^\gamma}$ and $|v|_{H^\gamma}$ are equivalent in $\check{H}^\gamma$.

We cite the well-posedness and smoothing properties of problem \eqref{E1}--\eqref{E3}.
\begin{lemma}\label{T-u}\cite{Zheng-2020, Zheng-2019}
If \textbf{Conditions A} and \textbf{B} hold and $\varphi\in \check{H}^{\gamma+2}$, $f\in H^1(0,T;\check{H}^{\gamma})$ for $\gamma>1/2$. Then problem \eqref{E1}--\eqref{E3} has unique solution $u\in C^1\big([0,T];\check{H}^{\gamma}\big)$ and
\begin{align*}
\|u\|_{C^1([0,T];\check{H}^s)}\leq c\big(\|\varphi\|_{\check{H}^{s+2}}+\|f\|_{H^1(0,T;\check{H}^s)}\big),\ \ 0\leq s \leq \gamma.
\end{align*}
\end{lemma}

\begin{lemma}\cite{Zheng-2020, Zheng-2019}
If \textbf{Conditions A} and \textbf{B} hold, $\alpha(t) \in C^1[0,T]$ and $\varphi\in \check{H}^{s+4}$, $f\in H^1(0,T;\check{H}^{s+2})\bigcap H^2(0,T;\check{H}^s)$ for $s\geq0$.
If $\alpha(0)>0$, then $u\in C^2\big([0,T];\check{H}^{s}\big)$ and for $0<\theta \ll 1$,
\begin{align*}
\|u\|_{C^2([\theta,T];\check{H}^s)}\leq c\theta^{-\alpha(0)}\big(\|\varphi\|_{\check{H}^{s+4}}+\|f\|_{H^1(0,T;\check{H}^{s+2})}+\|f\|_{H^2(0,T;\check{H}^s)}\big);
\end{align*}
if $\alpha(0)=0$, then $u\in C^2\big([0,T];\check{H}^{s}\big)$ and
\begin{align*}
\|u\|_{C^2([0,T];\check{H}^s)}\leq c\big(\|\varphi\|_{\check{H}^{s+4}}+\|f\|_{H^1(0,T;\check{H}^{s+2})}+\|f\|_{H^2(0,T;\check{H}^s)}\big).
\end{align*}
\end{lemma}

To be specific, in addition to the assumption that the data in \eqref{E1}--\eqref{E3} are sufficiently smooth, the condition $\alpha(0)=0$ admits a temporal $C^2$ solution with respect to the spatial norm $\|\cdot\|_{\check{H}^s(\Omega)}$ for large $s$ defined by the eigenpairs of the diffusion operator in \eqref{E1}--\eqref{E3}. Otherwise, if $\alpha(0)>0$, then $u_{tt}\in C(0,T]$ satisfies the pointwise-in-time estimate $\|u_{tt}\|_{\check{H}^s(\Omega)}\leq ct^{-\alpha(0)}$, which indicates the singularity as that of the solutions to the CO tFDEs.

\subsection{Basic approximations to VO Caputo fractional derivative}
In this section, we first recall $L1$ formula and  $F$-$L1$ formula for the VO Caputo fractional derivative (\ref{caputo}) with $t\in[0,T]$. For a positive integer $n$, let $\Delta t=T/n$ be the time step and we further define $t_k=k\Delta t$ for $k=0,1,\ldots,n$. At each time level $t_k$, we have
\begin{align}\label{caputo}
\prescript{C}{0}{\mathcal{D}}^{\alpha(t_k)}_tu(x,t_k)&=\frac 1 {\Gamma\big(1-\alpha(t_k)\big)}\int_{0}^{t_{k}}\frac{u'(x,\tau)}{(t_k-\tau)^{\alpha(t_k)}}\mathrm{d}\tau.
\end{align}
For convenience, denote $\alpha_k=\alpha(t_k)$. To discretize the VO Caputo fractional derivative, denote the linear interpolation for $u(x,\tau)$ over the interval $[t_{k-1},t_k]$ with $1\leq k\leq n$ by
\begin{align*}
L_{1,k} (\tau)=\frac{t_k-\tau}{\Delta t}u(x,t_{k-1})+\frac{\tau-t_{k-1}}{\Delta t}u(x,t_k),\ \ \tau\in[t_{k-1},t_k],
\end{align*}
giving a first-order approximation to $u'(x,\tau)$ on $[t_{k-1},t_k]$ by
\begin{align*}
L_{1,k}' (\tau)=\frac{u(x,t_k)-u(x,t_{k-1})}{\Delta t}.
\end{align*}
Then, the piecewise approximating function is defined by
\begin{align*}
L_1 (\tau)=\left\{L_{1,k}(\tau)|\tau\in[t_{k-1},t_k], k=1,2,\ldots,n\right\}.
\end{align*}
Thus,  $L1$  formula to (\ref{caputo}) is obtained as \cite{Zheng-2020}
\begin{align}
\prescript{}{0}{\mathcal{D}}^{\alpha_k}_tu(x,t_k)=&\frac 1 {\Gamma(1-\alpha_k)}\int_{0}^{t_k}\frac{L_1'(\tau)}{(t_k-\tau)^{\alpha_k}}\mathrm{d}\tau\nonumber\\
=&\frac {\Delta t^{-\alpha_k}}{\Gamma(2-\alpha_k)}
  \bigg(a_0^{(k)}u(x,t_k)-\sum\limits_{l=1}^{k-1}\Big(a_{k-l-1}^{(k)}-a_{k-l}^{(k)}\Big)u(x,t_l)-a_{k-1}^{(k)}u(x,t_0)\bigg),\label{L1}
\end{align}
where $a_l^{(k)}=(l+1)^{1-\alpha_k}-l^{1-\alpha_k}$. The truncation error of $L1$ formula is estimated by the following lemma.
\begin{lemma}\label{truncation-L1}(see \cite{Zheng-2020})
Suppose {\bf Conditions A} and {\bf B} hold, $\alpha(t) \in C^1[0,T]$ and $\varphi\in \check{H}^4$, $f\in H^1(0,T;\check{H}^2)\bigcap H^2(0,T;L_2)$. Let the VO Caputo fractional derivative at $t_k$ be as (\ref{caputo}) and $L1$ formula be defined by (\ref{L1}).
If $\alpha(0)>0$,
\begin{align*}
\Big\|\prescript{C}{0}{\mathcal{D}}^{\alpha_k}_tu(x,t_k)-\prescript{}{0}{\mathcal{D}}^{\alpha_k}_tu(x,t_k)\Big\|\leq ck^{-\alpha^*}\Delta t^{1-\alpha^*};
\end{align*}
if $\alpha(0)=0$,
\begin{align*}
\Big\|\prescript{C}{0}{\mathcal{D}}^{\alpha(t)}_tu(x,t)-\prescript{}{0}{\mathcal{D}}^{\alpha(t)}_tu(x,t)\Big\|_{\widehat{L}_\infty(0,T;L_2)}
:=\max_{1\leq k\leq n}\Big\|\prescript{C}{0}{\mathcal{D}}^{\alpha_k}_tu(x,t_k)-\prescript{}{0}{\mathcal{D}}^{\alpha_k}_tu(x,t_k)\Big\|\leq &c\Delta t.
\end{align*}
\end{lemma}

Due to the nonlocality of the VO fractional derivative, using $L1$ formula to calculate the value at the current time level, it needs to compute the sum of a series including the values of all previous time levels. Therefore, $L1$ formula requires large memory and computational cost. A fast algorithm (denoted by $F$-$L1$ formula) is developed in \cite{Zhang-2021} to discretize the VO Caputo fractional derivative \eqref{caputo-def}. For the expected accuracy $\epsilon$, $F$-$L1$ formula for the VO Caputo fractional derivative is given by
\begin{align*}
\prescript{\mathcal{F}}{0}{\mathcal{D}}^{\alpha_k}_tu(x,t_k)
=&\frac {T^{-\alpha_k}}{\Gamma(1-\alpha_k)}
  \sum_{i=\underline{N}+1}^{\overline{N}}\widetilde{\theta}_i^{(k)}\widetilde{F}_{k,i}+\frac{u(x,t_{k})-u(x,t_{k-1})}{\Delta t^{\alpha_k}\Gamma(2-\alpha_k)},
\end{align*}
where $\widetilde{F}_{1,i}=0$ and
\begin{align*}
\widetilde{F}_{k,i}=e^{-\widetilde{\lambda}_i\Delta t/T} \widetilde{F}_{k-1,i}
                    +T\frac{e^{-\widetilde{\lambda}_i\Delta t /T}-e^{-2\widetilde{\lambda}_i\Delta t/T}}{\widetilde{\lambda}_i\Delta t}\Big(u(x,t_{k-1})-u(x,t_{k-2})\Big),\ \ k=2,3,\ldots,n,
\end{align*}
in which the quadrature exponents and weights are given by
\begin{align*}
\widetilde{\lambda}_i=e^{i\widetilde{h}},\ \ \widetilde{\theta}_i^{(k)}=\frac{\widetilde{h}e^{\alpha_ki\widetilde{h}}}{\Gamma(\alpha_k)}
\end{align*}
with
\begin{align}\label{parameter-F}
\widetilde{h}&=\frac {2\pi}{\log3+\alpha^*\log(\cos 1)^{-1}+\log\epsilon^{-1}},\nonumber\\
\underline{N}&=\left\lceil\frac 1{\widetilde{h}}\frac 1{\alpha_*}\big(\log\epsilon+\log\Gamma(1+\alpha^*)\big)\right\rceil,\\
\overline{N}&=\left\lfloor\frac 1{\widetilde{h}}\left(\log \frac T{\Delta t}+\log\log\epsilon^{-1}+\log\alpha_*+2^{-1}\right)\right\rfloor.\nonumber
\end{align}
In addition,
\begin{align*}
\prescript{\mathcal{F}}{0}{\mathcal{D}}^{\alpha_1}_tu(x,t_1)=\frac{u(x,t_1)-u(x,t_0)}{\Delta t^{\alpha_1}\Gamma(2-\alpha_1)}.
\end{align*}

Compared with $L1$ formula, $F$-$L1$ formula reduces the storage requirement from $\mO(n)$ to $\mO(\log^2 n)$ and the computational cost from $\mO(n^2)$ to $\mathcal{O}(n\log^2 n)$. It provides an efficient tool to approximate the VO Caputo fractional derivative. However, the fast method proposed in \cite{Zhang-2021} cannot deal with the problem with a small $\alpha_*$. In fact, a small value of $\alpha_*$ requires a large $\underline{N}$ and even $\underline{N} \rightarrow \infty$ as $\alpha_* \rightarrow 0$, which affects the accuracy of the approximation.
In order to overcome the shortage, we develop a robust fast approximation, which is called $RF$-$L1$ formula for the VO Caputo fractional derivative.

\section{$RF$-$L1$ formula to VO Caputo fractional derivative}\label{fast-formula}
We first split the integral in \eqref{L1} into two parts. Then that can be decomposed as
\begin{align}\label{HL}
\prescript{C}{0}{\mathcal{D}}^{\alpha_k}_tu(x,t_k)
&\approx \frac 1 {\Gamma(1-\alpha_k)}\int_{0}^{t_{k-1}}\frac{L_{1}'(\tau)}{(t_k-\tau)^{\alpha_k}}\mathrm{d}\tau
         +\frac 1 {\Gamma(1-\alpha_k)}\int_{t_{k-1}}^{t_k}\frac{L_{1,k}'(\tau)}{(t_k-\tau)^{\alpha_k}}\mathrm{d}\tau\nonumber\\
&:= I_{his}(t_k)+I_{loc}(t_k),
\end{align}
where we call $I_{his}(t_k)$ and $I_{loc}(t_k)$ the history and local parts, respectively. Since the local part contributes few memory and computational cost compared with the history part, we keep the local part be as in \eqref{HL}. For the history part, noting that $I_{his}(t_1)=0$, we have
\begin{align}\label{RF0}
\prescript{\mathcal{RF}}{0}{\mathcal{D}}^{\alpha_1}_tu(x,t_1)
:=I_{loc}(t_1)=\prescript{ }{0}{\mathcal{D}}{}^{\alpha_1}_tu(x,t_1).
\end{align}
Using the integration by parts for $k=2,3,\ldots,n$, we have
\begin{align}\label{his}
I_{his}(t_k)
=&\frac{1}{\Gamma(1-\alpha_k)}\bigg( \frac{u(x,t_{k-1})}{\Delta t^{\alpha_k}}
  -\frac{u(x,t_0)}{t_k^{\alpha_k}}-\alpha_k\int_0^{t_{k-1}}\frac{L_{1}(\tau)}{(t_k-\tau)^{1+\alpha_k}}\mathrm{d}\tau\bigg)\nonumber\\
=&\frac{1}{\Gamma(1-\alpha_k)}\bigg( \frac{u(x,t_{k-1})}{\Delta t^{\alpha_k}}-\frac{u(x,t_0)}{t_k^{\alpha_k}}
         -\frac{\alpha_k}{T^{1+\alpha_k}}\int_0^{t_{k-1}} L_1(\tau) \Big(\frac{t_k-\tau} T\Big)^{-1-\alpha_k}\mathrm{d}\tau\bigg).
\end{align}
Noting that $1+\alpha_k>0$ and $0<\Delta t/T\leq (t_k-\tau)/T$ for $\tau\in[0,t_{k-1}]$, with the help of ESA technique \cite{Beylkin-2017, Zhang-2021}, the kernel $\big((t_k-\tau)/ T\big)^{-(1+\alpha_k)}$ in (\ref{his}) can be approximated using a linear combination of exponentials. We have the following lemma to structure the robust fast formula.
\begin{lemma}\cite{Beylkin-2017, Zhang-2021}\label{L-soe}
For any constant $\beta_k\in[\beta_*,\beta^*]\subset[1,2)$, $0< \Delta t/T\leq (t_k-\tau)/T\leq1$ for $\tau\in[0,t_{k-1}]$, $1\leq k \leq n-1$ and the expected accuracy $0<\epsilon\leq\ 1/e$, there exist a constant $h$, integers $N^*$ and $N_*$, which satisfy
\begin{align}\label{parameter-RF}
&h=\frac {2\pi}{\log3+\beta^*\log(\cos 1)^{-1}+\log\epsilon^{-1}},\nonumber\\
&N^*=\left\lceil\frac 1{h}\frac 1{\beta_*}\big(\log\epsilon+\log\Gamma(1+\beta^*)\big)\right\rceil,\\
&N_*=\left\lfloor\frac 1{h}\Big(\log \frac T{\Delta t}+\log\log\epsilon^{-1}+\log\beta_*+2^{-1}\Big)\right\rfloor,\nonumber
\end{align}
such that
\begin{align*}
\Bigg|\Big(\frac {t_k-\tau}{T}\Big)^{-\beta_k}-\sum_{i=N_*+1}^{N^*}\theta_i^{(k)} e^{-\lambda_i(t_k-\tau)/T}\Bigg|
\leq \epsilon \Big(\frac {t_k-\tau}{T}\Big)^{-\beta_k},
\end{align*}
where the quadrature exponents and weights are given by
\begin{align*}
\lambda_i=e^{ih},\ \ \theta_i^{(k)}=\frac{he^{\beta_kih}}{\Gamma(\beta_k)}. 
\end{align*}
Furthermore, the total number of terms in the summation can be estimated as
\begin{align*}
N_\epsilon=N^*-N_*
          \leq \frac 1 {10}\Big(2\log\frac 1 \epsilon+\log\beta^*+2\Big)
               \Big(\log \frac T {\Delta t}+\frac 1 \beta_*\log\frac 1 \epsilon+\log\log\frac 1 \epsilon+\frac 3 2\Big).
\end{align*}
\end{lemma}

So invoking  Lemma \ref{L-soe} with $\beta_k=1+\alpha_k$ in \eqref{his}, we obtain
\begin{align}\label{app-his}
&I_{his}(t_k)\nonumber\\
\approx&\frac 1{\Gamma(1-\alpha_k)}\bigg(\frac{u(x,t_{k-1})}{\Delta t^{\alpha_k}}-\frac{u(x,t_{0})}{t_k^{\alpha_k}}
        -\frac{\alpha_k}{T^{1+\alpha_k}}\sum_{i=N_*+1}^{N^*}\theta_i^{(k)}\int_0^{t_{k-1}}L_1(\tau)e^{-(t_k-\tau)/T\lambda_i}\mathrm{d}\tau\bigg),
\end{align}
where the quadrature weights and points are defined by
\begin{align*}
\lambda_i=e^{ih}, \ \ \theta_i^{(k)}=\frac{he^{(1+\alpha_k) ih}}{\Gamma(1+\alpha_k)},
\end{align*}
in which $h$, $N_*$ and $N^*$ are defined by \eqref{parameter-RF}.
The related discretization formula is obtained (later in Lemma \ref{L-truncation-L1-RF}) as
\begin{align}\label{his-RF}
I_{his}(t_k)=I_{his,\epsilon}(t_k)+\mathcal{O}(\epsilon\Delta t^{-\alpha_k}),
\end{align}
where $I_{his,\epsilon}(t_k)$ is defined by
\begin{align}\label{his-fast}
I_{his,\epsilon}(t_k)
=\frac 1{\Gamma(1-\alpha_k)}\bigg(\frac{u(x,t_{k-1})}{\Delta t^{\alpha_k}}-\frac{u(x,t_0)}{t_k^{\alpha_k}}
 -\frac{\alpha_k}{T^{1+\alpha_k}}\sum_{i=N_*+1}^{N^*}\theta_i^{(k)}F_{k,i}\bigg),
\end{align}
in which $F_{k,i}$ is given by
\begin{align}\label{y}
F_{k,i}=\int_0^{t_{k-1}}L_1(\tau)e^{-(t_k-\tau)\lambda_i/T}\mathrm{d}\tau.
\end{align}
Note that $F_{1,i}=0$ and $F_{k,i}$ can be calculated by the following recursive relation for $k=2,3,\ldots,n,$
\begin{align}\label{y-compute}
F_{k,i}
=&e^{-\Delta t\lambda_i/T} F_{k-1,i}+\int_{t_{k-2}}^{t_{k-1}}L_{1,k-1}(\tau)e^{-(t_k-\tau)\lambda_i/T}\mathrm{d}\tau\nonumber\\
=&e^{-\Delta t\lambda_i/T} F_{k-1,i}\nonumber\\
 &+T\frac{e^{-\Delta t\lambda_i/T}}{\Delta t\lambda_i^2}\Big(-\Delta t\lambda_i e^{-\Delta t\lambda_i/T}+T-Te^{-\Delta t\lambda_i/T}\Big)u(x,t_{k-2})\nonumber\\
 &+T\frac{e^{-\Delta t\lambda_i/T}}{\Delta t\lambda_i^2}\Big(\Delta t\lambda_i-T+Te^{-\Delta t\lambda_i/T}\Big)u(x,t_{k-1}).
\end{align}

According to \eqref{his-RF}, replacing $I_{his}(t_k)$ in (\ref{HL}) by (\ref{his-fast}), we obtain $RF$-$L1$ formula as
\begin{align}\label{RF1}
\prescript{\mathcal{RF}}{0}{\mathcal{D}}^{\alpha_k}_tu(x,t_k)
=&I_{his,\epsilon}(t_k)+I_{loc}(t_k)\nonumber\\
=&\frac 1 {\Gamma(1-\alpha_k)}\bigg(\frac{u(x,t_{k-1})}{\Delta t^{\alpha_k}}
  -\frac{u(x,t_0)}{t_k^{\alpha_k}}-\frac{\alpha_k}{T^{1+\alpha_k}}\sum_{i=N_*+1}^{N^*}\theta_i^{(k)}F_{k,i}\bigg)\nonumber\\
 &+\frac{u(x,t_{k})-u(x,t_{k-1})}{\Delta t^{\alpha_k}\Gamma(2-\alpha_k)},\ \ k=2,3,\ldots,n.
\end{align}
Recalling (\ref{RF0}), we have
\begin{align}\label{RF2}
\prescript{\mathcal{RF}}{0}{\mathcal{D}}^{\alpha_1}_tu(x,t_1)
=\frac{u(x,t_1)-u(x,t_0)}{\Delta t^{\alpha_1}\Gamma(2-\alpha_1)}.
\end{align}
Summarizing all this activity, we give the following algorithm to show the detailed instruction for the implementation of the robust fast algorithm for approximating the VO Caputo fractional derivative.

\begin{algorithm}[H]
\caption{Robust fast algorithm to approximate VO Caputo fractional derivative gradually}
\label{alg-fast}
\begin{algorithmic}[1]
\State Give the time step $\Delta t$, the expected accuracy $\epsilon$ and set $h,N_*,N^*$ correspondingly
\State Compute $\prescript{\mathcal{RF}}{0}{\mathcal{D}}_t^{\alpha_1}u(x,t_1)$ by formula (\ref{RF2})
\State Set $\{F_{1,i}=0\}_{i=N_*+1}^{N^*}$ and $\left\{\lambda_i=e^{ih}\right\}_{i=N_*+1}^{N^*}$
\For{$k=2,3,\ldots,n$}
\State Set $\left\{\theta_i^{(k)}=\frac{he^{(1+\alpha_k)ih}}{\Gamma(1+\alpha_k)}\right\}_{i=N_*+1}^{N^*}$ and update $\{F_{k,i}\}_{i=N_*+1}^{N^*}$ by formula (\ref{y-compute}) 
\State Compute $\prescript{\mathcal{RF}}{0}{\mathcal{D}}_t^{\alpha_k}u(x,t_k)$ by formula (\ref{RF1}) using $\{F_{k,i},\theta_i^{(k)},\lambda_i\}_{i=N_*+1}^{N^*}$
\EndFor
\end{algorithmic}
\end{algorithm}

\begin{remark}\label{memory}
At each time level, it only needs $\mathcal{O}(1)$ computational cost to compute $F_{k,i}$ since $F_{k-1,i}$ is known in advance. In total the robust fast algorithm requires only $\mathcal{O}(\log^2n)$ memory and $\mathcal{O}(n\log^2n)$ computational cost when numerically discretize the VO Caputo fractional derivative. The proposed method provides an efficient tool to approximate the VO Caputo fractional derivative.
\end{remark}

\subsection{Truncation error}\label{truncation}
In this subsection, we study the truncation error of $RF$-$L1$ formula (\ref{RF1})--(\ref{RF2}) to the VO Caputo fractional derivative. To investigate the truncation error of $RF$-$L1$ formula to the VO Caputo fractional derivative, we first give the following lemma to state the error bound of $RF$-$L1$ formula to $L1$ formula.
\begin{lemma}\label{L-truncation-L1-RF}
Suppose \textbf{Conditions A} and \textbf{B} hold, $\alpha(t) \in C^1[0,T]$, $\varphi\in \check{H}^4$, $f\in H^1(0,T;\check{H}^2)$ $\bigcap H^2(0,T;L_2)$. Let $L1$ formula be as (\ref{L1}), $RF$-$L1$ formula be defined by (\ref{RF1})--(\ref{RF2}), $\epsilon$ be the expected accuracy, then
\begin{align}\label{error-L1-RF}
\Big\|\prescript{}{0}{\mathcal{D}}^{\alpha_k}_tu(x,t_k)-\prescript{\mathcal{RF}}{0}{\mathcal{D}}^{\alpha_k}_tu(x,t_k)\Big\|
=\mathcal{O}\left(\epsilon\Delta t^{-\alpha_k}\right).
\end{align}
\end{lemma}
\begin{proof}
For $k=1$, we obtain the lemma directly by (\ref{RF0}). For $k=2,3,\ldots,n$, obviously, the only difference between $\prescript{\mathcal{RF}}{0}{\mathcal{D}}^{\alpha_k}_tu(x,t_k)$  and $\prescript{}{0}{\mathcal{D}}^{\alpha_k}_tu(x,t_k)$ is the approximation to the history part. According to \eqref{his} and \eqref{app-his}, the error can be estimated by
\begin{align*}
&\Big\|\prescript{}{0}{\mathcal{D}}^{\alpha_k}_tu(x,t_k)-\prescript{\mathcal{RF}}{0}{\mathcal{D}}^{\alpha_k}_tu(x,t_k)\Big\|\\
=&\frac {\alpha_k}{T^{1+\alpha_k}\Gamma(1-\alpha_k)}\left\| \int_{0}^{t_{k-1}} L_1(\tau)\bigg(\sum\limits_{i=N_*+1}^{N^*}\theta_i^{(k)}e^{-(t_k-\tau)\lambda_i/T}
  -\Big(\frac{t_k-\tau}T\Big)^{-1-\alpha_k}\bigg)\mathrm{d}\tau\right\|.
\end{align*}
By Lemma \ref{L-soe}, we have
\begin{align*}
\Big\|\prescript{}{0}{\mathcal{D}}^{\alpha_k}_tu(x,t_k)-\prescript{\mathcal{RF}}{0}{\mathcal{D}}^{\alpha_k}_tu(x,t_k)\Big\|
\leq&\frac {\alpha_k}{\Gamma(1-\alpha_k)}\epsilon\left\|\int_{0}^{t_{k-1}}\frac{L_1(\tau)}{(t_k-\tau)^{1+\alpha_k}}\mathrm{d}\tau\right\|\\
\leq&\frac {1}{\Gamma(1-\alpha_k)}\epsilon \|u\|_{C([0,T];\check{H}^s)}\Delta t^{-\alpha_k} \\
=& \mathcal{O}\left(\epsilon\Delta t^{-\alpha_k}\right),
\end{align*}
where we apply Lemma \ref{T-u}, which guarantees the boundedness of $\|u\|_{C([0,T];\check{H}^s)}$. The proof is completed.
\end{proof}

We obtain the following theorem to estimate the truncation error of $RF$-$L1$ formula to the VO Caputo farctional derivative.
\begin{theorem}\label{truncation-error}
Suppose \textbf{Conditions A} and \textbf{B} hold, $\alpha(t) \in C^1[0,T]$, $\varphi\in \check{H}^4$, $f\in H^1(0,T;\check{H}^2)$ $\bigcap$ $ H^2(0,T;L_2)$.  Let the VO Caputo fractional derivative at $t_k$ be as (\ref{caputo}), its $RF$-$L1$ formula be defined by (\ref{RF1})--(\ref{RF2}) and $\epsilon\leq\mathcal{O}(\Delta t^{1+\alpha^*})$ be the expected accuracy.
If $\alpha(0)>0$,
\begin{align*}
\Big\|\prescript{C}{0}{\mathcal{D}}^{\alpha_k}_tu(x,t_k)-\prescript{\mathcal{RF}}{0}{\mathcal{D}}^{\alpha_k}_tu(x,t_k)\Big\|
\leq ck^{-\alpha^*}\Delta t^{1-\alpha^*};
\end{align*}
if $\alpha(0)=0$,
\begin{align*}
\left\|\prescript{C}{0}{\mathcal{D}}^{\alpha(t)}_tu(x,t)-\prescript{\mathcal{RF}}{0}{\mathcal{D}}^{\alpha(t)}_tu(x,t)\right\|_{\widehat{L}_\infty(0,T;L_2)}\leq c\Delta t.
\end{align*}

\end{theorem}
\begin{proof}
The triangle inequality leads to
\begin{align*}
\Big\|\prescript{C}{0}{\mathcal{D}}^{\alpha_k}_tu(x,t_k)-\prescript{\mathcal{RF}}{0}{\mathcal{D}}^{\alpha_k}_tu(x,t_k)\Big\|
\leq & \Big\|\prescript{C}{0}{\mathcal{D}}^{\alpha_k}_tu(x,t_k)-\prescript{}{0}{\mathcal{D}}^{\alpha_k}_tu(x,t_k)\Big\|\\
     &+ \Big\|\prescript{}{0}{\mathcal{D}}^{\alpha_k}_tu(x,t_k)-\prescript{\mathcal{RF}}{0}{\mathcal{D}}^{\alpha_k}_tu(x,t_k)\Big\|.
\end{align*}
The desired result now follows on recalling Lemma \ref{truncation-L1} and  Lemma \ref{L-truncation-L1-RF}.
\end{proof}

\subsection{Properties of discrete kernels}
To simplify the further study of the convergence of the finite difference scheme, we first rewrite (\ref{RF1})--(\ref{RF2}) using (\ref{y}) into another form. For convenience, we denote \begin{align*}
s^{(k)}=\frac {\Delta t^{-\alpha_k}}{\Gamma(2-\alpha_k)},\ \ k=1,2,\ldots,n.
\end{align*}
By rearranging, it can be rewritten as
\begin{align}\label{re-fast}
\prescript{\mathcal{RF}}{0}{\mathcal{D}}^{\alpha_k}_tu(x,t_k)
=s^{(k)}\Big(u(x,t_k)-\sum_{l=0}^{k-1}d_l^{(k)}u(x,t_l)\Big),\ \ k=1,2,\ldots,n,
\end{align}
where
\begin{align*}
&d_0^{(k)}=(1-\alpha_k)\bigg(k^{-\alpha_k}+\frac{\alpha_k}{T^{1+\alpha_k}} \Delta t^{\alpha_k-1} \sum_{i=N_*+1}^{N^*}\theta_i^{(k)}\int_{t_0}^{t_1}(t_1-\tau)e^{-\lambda_i(t_k-\tau)/T} \mathrm{d}\tau\bigg),\\
&d_{k-1}^{(k)}=\alpha_k+\frac{\alpha_k(1-\alpha_k)}{T^{1+\alpha_k}} \Delta t^{\alpha_k-1} \sum_{i=N_*+1}^{N^*}\theta_i^{(k)}\int_{t_{k-2}}^{t_{k-1}}(\tau-t_{k-2})e^{-\lambda_i(t_k-\tau)/T} \mathrm{d}\tau,
\end{align*}
and
\begin{align*}
d_l^{(k)}
=&\frac{\alpha_k(1-\alpha_k)}{T^{1+\alpha_k}} \Delta t^{\alpha_k-1}\sum_{i=N_*+1}^{N^*}\theta_i^{(k)}\bigg(\int_{t_{l-1}}^{t_l}(\tau-t_{l-1})e^{-\lambda_i(t_k-\tau)/T} \mathrm{d}\tau\\
 &+\int_{t_{l}}^{t_{l+1}}(t_{l+1}-\tau)e^{-\lambda_i(t_k-\tau)/T} \mathrm{d}\tau\bigg),\ \ l=1,2,\ldots,k-2.
\end{align*}
In particular $d_0^{(1)}=1$.

Meanwhile, these coefficients satisfy the following lemma.
\begin{lemma}\label{coe}
Let $\{d_l^{(k)}\}_{l=0}^{k-1} (k=1,2,\ldots,n)$ be defined by (\ref{re-fast}) and $\epsilon$ be the expected accuracy. Then, we have
\begin{align*}
&(1)d_l^{(k)}>0;\\
&(2)\sum_{l=0}^{k-1}d_l^{(k)}\leq 1+\epsilon.
\end{align*}
\end{lemma}
\begin{proof}
(1) This conclusion can be obtained by a straight forward calculation.\\
(2) Summing up $d_l^{(k)}$ for $l$ from $0$ to $k-1$, and rearranging the integral terms, we have
\begin{align*}
\sum_{l=0}^{k-1}d_l^{(k)}
=&(1-\alpha_k)k^{-\alpha_k}+\alpha_k\\
 &+\frac{\alpha_k(1-\alpha_k)}{T^{1+\alpha_k}}\Delta t^{\alpha_k-1}\sum_{j=0}^{k-2}\int_{t_j}^{t_{j+1}}(t_{j+1}-t_j)
  \sum_{i=N_*+1}^{N^*}\theta_i^{(k)}e^{-\lambda_i(t_k-\tau)}\mathrm{d}\tau\\
=&(1-\alpha_k)k^{-\alpha_k}+\alpha_k+\frac{\alpha_k(1-\alpha_k)}{T^{1+\alpha_k}}\Delta t^{\alpha_k}\int_{0}^{t_{k-1}}
  \sum_{i=N_*+1}^{N^*}\theta_i^{(k)}e^{-\lambda_i(t_k-\tau)}\mathrm{d}\tau\\
\leq &(1-\alpha_k)k^{-\alpha_k}+\alpha_k+(1+\epsilon)(1-\alpha_k)\alpha_k\Delta t^{\alpha_k}\int_{0}^{t_{k-1}}(t_k-\tau)^{-1-\alpha_k}\mathrm{d}\tau\\
=&(1-\alpha_k)k^{-\alpha_k}+\alpha_k+(1+\epsilon)(1-\alpha_k)(1-k^{-\alpha_k})\\
=&1+\epsilon-\alpha_k\epsilon+\alpha_k\epsilon k ^{-\alpha_k}- \epsilon k ^{-\alpha_k}\\
\leq &1+\epsilon,
\end{align*}
where we use the estimate
\begin{align*}
(1-\epsilon)t^{-1-\alpha_k}
\leq \sum_{i=N_*+1}^{N^*}\theta_i^{(k)}e^{-\lambda_it}
\leq (1+\epsilon)t^{-1-\alpha_k}.
\end{align*}
The proof is completed.
\end{proof}

\section{Finite difference scheme for VO tFDE}\label{finite-difference-scheme}
In this section, we propose a fast finite difference scheme for solving \eqref{E1}--\eqref{E3} and prove its error estimates.

We first discretize the first-order partial derivative $u_t$ by
\begin{align}\label{ut}
\frac {\partial}{\partial t}u(x,t_k)
=\frac{u(x,t_k)-u(x,t_{k-1})}{\Delta t}+E^k :=\Delta_tu(x,t_k)+E^k,
\end{align}
where $E^k$ satisfies the following lemma.
\begin{lemma}\label{center-err}\cite{Zheng-2020}
Suppose \textbf{Conditions A} and \textbf{B} hold, $\alpha(t) \in C^1[0,T]$, $\varphi\in \check{H}^4$ and $f\in H^1(0,T;\check{H}^2)\bigcap H^2(0,T;L_2)$. If $\alpha(0)>0$,
 \begin{align*}
\big\|E^k\big\|
\leq ck^{-\alpha(0)}\Delta t^{1-\alpha(0)};
\end{align*}
if $\alpha(0)=0$,
\begin{align*}
\big\|E\big\|_{\widehat{L}_\infty(0,T;L_2)}
\leq c\Delta t.
\end{align*}
\end{lemma}

Let $m$ be a positive integer, $\Delta x=(x_r-x_l)/m$ be the size of spatial grid, and define a spatial partition $x_j=x_l+j\Delta x$ for $j=0,1,\ldots,m$. Denote $x_{j+1/2}=(x_{j+1}+x_j)/2$ as the midpoint of the neighboring nodes $x_{j+1}$ and $x_j$, $p_{j\pm1/2}=p(x_{j\pm1/2})$, $f_j^k=f(x_j,t_k)$ and $\varphi_j=\varphi(x_j)$.

The integer-order diffusion term is discretized by \cite{T-2013}
\begin{align}\label{spatial-derivative}
&\frac \partial{\partial x}\bigg[p(x_j)\frac {\partial u(x_j,t_k)}{\partial x}\bigg]\\
=&\frac 1 {\Delta x}\bigg(p_{j+1/2}\frac{u(x_{j+1},t_k)-u(x_j,t_k)}{\Delta x}-p_{j-1/2}\frac{u(x_j,t_k)-u(x_{j-1},t_k)}{\Delta x}\bigg)+G_j^k\nonumber\\
:=&\Delta_xu(x_j,t_k)+G_j^k,\nonumber
\end{align}
where $G_j^k=c\Delta x^2.$

Then substituting \eqref{ut}, \eqref{spatial-derivative} and \eqref{re-fast} into \eqref{E1} gives
\begin{align}
&\Delta_tu(x_j,t_k)+\zeta\prescript{\mathcal{RF}}{0}{\mathcal{D}}^{\alpha_k}_tu(x_j,t_k)\nonumber\\
=&\Delta_xu(x_j,t_k)+f_j^k-\big(E^k(x_j)+\zeta R^k(x_j)-G_j^k\big),\ \ 1\leq j \leq m-1,\ \ 1\leq k \leq n.\label{S1}
\end{align}
From the initial and boundary value conditions \eqref{E2}--\eqref{E3}, we have
\begin{align}
&u(x_j,0)=\varphi_j,\ \ 0\leq j \leq m,\label{S2}\\
&u(x_l,t_k)=u(x_r,t_k)=0,\ \ 1\leq k \leq n.\label{S3}
\end{align}

\subsection{Finite difference schemes}
Denote the approximate solution to $u(x_j,t_k)$ by $U_j^k$. Then,
\begin{align*}
\Delta_tU_j^k&=\frac{U_j^k-U_j^{k-1}}{\Delta t},\\
\Delta_xU_j^k&=\frac 1 {\Delta x}\bigg(p_{j+1/2}\frac{U_{j+1}^k-U_j^k}{\Delta x}-p_{j-1/2}\frac{U_j^k-U_{j-1}^k}{\Delta x}\bigg).
\end{align*}
We obtain $RF$-$L1$ scheme for the problem \eqref{E1}--\eqref{E3} as follows
\begin{align}
&\Delta_tU_j^k+\zeta\prescript{\mathcal{RF}}{0}{\mathcal{D}}^{\alpha_k}_tU_j^k
=\Delta_xU_j^k+f_j^k,\ \ 1\leq j \leq m-1,\ \ 1\leq k \leq n,\label{RFS1}\\
&U_j^0=\varphi_j,\ \ 0\leq j \leq m,\label{RFS2}\\
&U_0^k=U_m^k=0,\ \ 1\leq k \leq n,\label{RFS3}
\end{align}
where
\begin{align*}
\prescript{\mathcal{RF}}{0}{\mathcal{D}}^{\alpha_k}_tU_j^k
=&\frac 1 {\Gamma(1-\alpha_k)}\bigg(\frac{U_j^{k-1}}{\Delta t^{\alpha_k}}
  -\frac{U_j^0}{t_k^{\alpha_k}}-\frac{\alpha_k}{T^{1+\alpha_k}}\sum_{i=N_*+1}^{N^*}\theta_i^{(k)}F_{k,i}\bigg)\nonumber\\
 &+\frac{U_j^k-U_j^{k-1}}{\Delta t^{\alpha_k}\Gamma(2-\alpha_k)},\ \ k=2,3,\ldots,n,\\
\prescript{\mathcal{RF}}{0}{\mathcal{D}}^{\alpha_1}_tU_j^1
=&\frac{U_j^1-U_j^0}{\Delta t^{\alpha_1}\Gamma(2-\alpha_1)},
\end{align*}
in which $F_{1,i}=0$ and for $k=2,3,\ldots$,
\begin{align*}
F_{k,i}
=&e^{-\Delta t\lambda_i} F_{k-1,i}+T\frac{e^{-\Delta t\lambda_i}}{\Delta t\lambda_i^2}
  \Big(-\Delta t\lambda_i e^{-\Delta t\lambda_i}+T-Te^{-\Delta t\lambda_i}\Big)U_j^{k-2}\nonumber\\
 &+T\frac{e^{-\Delta t\lambda_i}}{\Delta t\lambda_i^2}\Big(\Delta t\lambda_i-T+Te^{-\Delta t\lambda_i}\Big)U_j^{k-1}.
\end{align*}

Similarly, we obtain $L1$ scheme as
\begin{align}
&\Delta_tU_j^k+\zeta\prescript{\mathcal{}}{0}{\mathcal{D}}^{\alpha_k}_tU_j^k
=\Delta_xU_j^k+f_j^k,\ \ 1\leq j \leq m-1,\ \ 1\leq k \leq n,\label{LS1}\\
&U_j^0=\varphi_j,\ \ 0\leq j \leq m,\label{LS2}\\
&U_0^k=U_m^k=0,\ \ 1\leq k \leq n,\label{LS3}
\end{align}
where
\begin{align*}
\prescript{ }{0}{\mathcal{D}}^{\alpha_k}_tU_j^k
=&\frac {\Delta t^{-\alpha_k}}{\Gamma(2-\alpha_k)}
  \bigg(a_0^{(k)}U_j^k-\sum\limits_{l=1}^{k-1}\Big(a_{k-l-1}^{(k)}-a_{k-l}^{(k)}\Big)U_j^l-a_{k-1}^{(k)}U_j^0\bigg),\nonumber
\end{align*}
in which $a_l^{(k)}=(l+1)^{1-\alpha_k}-l^{1-\alpha_k}$.

And $F$-$L1$ scheme  is obtained as
\begin{align}
&\Delta_tU_j^k+\zeta\prescript{\mathcal{F}}{0}{\mathcal{D}}^{\alpha_k}_tU_j^k
=\Delta_xU_j^k+f_j^k,\ \ 1\leq j \leq m-1,\ \ 1\leq k \leq n,\label{FS1}\\
&U_j^0=\varphi_j,\ \ 0\leq j \leq m,\label{FS2}\\
&U_0^k=U_m^k=0,\ \ 1\leq k \leq n,\label{FS3}
\end{align}
where
\begin{align*}
\prescript{\mathcal{F}}{0}{\mathcal{D}}^{\alpha_k}_tU_j^k
=&\frac {T^{-\alpha_k}}{\Gamma(1-\alpha_k)}\sum_{i=\underline{N}+1}^{\overline{N}}\widetilde{\theta}_i^{(k)}\widetilde{F}_{k,i}+\frac{U_j^k-U_j^{k-1}}{\Delta t^{\alpha_k}\Gamma(2-\alpha_k)},\ \ k=2,3,\ldots,n,\\
\prescript{\mathcal{F}}{0}{\mathcal{D}}^{\alpha_1}_tU_j^1=&\frac{U_j^1-U_j^0}{\Delta t^{\alpha_1}\Gamma(2-\alpha_1)},
\end{align*}
in which $\widetilde{F}_{1,i} =0$, and for $k=2,3,\ldots,$
\begin{align*}
\widetilde{F}_{k,i}
=&e^{-\widetilde{\lambda}_i\Delta t/T} \widetilde{F}_{k-1,i}
 +T\frac{e^{-\widetilde{\lambda}_i\Delta t /T}-e^{-2\widetilde{\lambda}_i\Delta t/T}}{\widetilde{\lambda}_i\Delta t}\big(U_j^{k-1}-U_j^{k-2}\big).
\end{align*}

\subsection{Convergence}\label{Stability and convergence}
Next we estimate the error without any artificial regularity assumptions of the true solution.
\begin{theorem}\label{T-stability}
Suppose \textbf{Conditions A} and \textbf{B} hold, $\alpha(t) \in C^1[0,T]$, $\varphi\in \check{H}^4$, $f\in H^1(0,T;\check{H}^2)$ $\bigcap H^2(0,T;L_2)$.  Suppose the expected accuracy $\epsilon\leq\mathcal{O}(\Delta t^{1+\alpha^*})$. Suppose $\{u(x_j,t_k)|0\leq j\leq m, 0\leq k\leq n\}$ and $\{U_j^k|0\leq j\leq m, 0\leq k\leq n\}$ are solutions of the problem (\ref{E1})--(\ref{E3}) and $RF$-$L1$ scheme (\ref{RFS1})--(\ref{RFS3}), respectively. Then
\begin{align*}
\|U-u\|_{\widehat{L}_\infty(0,T;L_2)}\leq c(\Delta t+\Delta x^2).
\end{align*}
\end{theorem}

\begin{proof}
Let $r_j^k=u_j^k-u(x_j,t_k)$ for $0\leq j\leq m$, $0\leq k\leq n$. Subtracting \eqref{RFS1}--\eqref{RFS3} from \eqref{S1}--\eqref{S3}, we obtain the error equation
\begin{align}
&\Delta_tr_j^k+\zeta\prescript{\mathcal{RF}}{0}{\mathcal{D}}^{\alpha_k}_tr_j^k
=\Delta_xr_j^k-\big(E^k(x_j)+\zeta R^k(x_j)-G_j^k\big), 1\leq j \leq m-1,1\leq k \leq n,\label{ER1}\\
&r_j^0=0,\ \ 0\leq j \leq m,\label{ER2}\\
&r_0^k=r_m^k=0,\ \ 1\leq k \leq n.\label{ER3}
\end{align}
Use $r_j^0=0$ and \eqref{re-fast} to rearrange \eqref{ER1} as
\begin{align*}
&\frac{r_j^k-r_j^{k-1}}{\Delta t}+\zeta s^{(k)}\bigg(r_j^k-\sum_{l=1}^{k-1}d_l^{(k)}r_j^l\bigg)
-\frac 1 {\Delta x}\bigg(p_{j+1/2}\frac{r_{j+1}^k-r_j^k}{\Delta x}-p_{j-1/2}\frac{r_j^k-r_{j-1}^k}{\Delta x}\bigg)\\
=&-\big(E^k(x_j)+\zeta R^k(x_j)-G_j^k\big).
\end{align*}
Making an inner product with $r^k$ on both hand sides of the equality, and from Theorem \ref{truncation-error}, Lemma \ref{center-err} and \eqref{spatial-derivative}, we obtain
\begin{align}\label{error}
\big(1+\zeta\Delta t s^{(k)}\big)\big\|r^k\big\|\leq \big\|r^{k-1}\big\|+\zeta\Delta t s^{(k)}\sum_{l=1}^{k-1}d_l^{(k)}\big\|r^l\big\|+\Delta tJ^k,
\end{align}
where $J^k\leq c_1k^{-\alpha(0)}n^{\alpha(0)-1}+c_2k^{-\alpha^*}n^{\alpha^*-1}+c_3\Delta x^2.$
It is clear from \eqref{error} and $\big\|r^0\big\|=0$ that
\begin{align*}
&\big\|r^1\big\|\leq\Delta tJ^1\leq\Delta t(1+\epsilon)J^1.
\end{align*}
Assume that
\begin{align}\label{error6}
&\big\|r^{k_0}\big\|\leq \Delta t(1+\epsilon)^{k_0}\sum\limits_{q=1}^{k_0}J^q,\ \ k_0=2,3,\ldots,k-1.
\end{align}
Using the mathematical induction, it is derived that
\begin{align*}
&\big(1+\zeta\Delta ts^{(k)}\big)\big\|r^k\big\|\nonumber\\
\leq&\Delta t(1+\epsilon)^{k-1}\sum\limits_{q=1}^{k-1}J^q
     +\zeta\Delta ts^{(k)}\sum_{l=1}^{k-1}d_l^{(k)}\Big(\Delta t(1+\epsilon)^l\sum\limits_{q=1}^lJ^q\Big)
      +\Delta tJ^k\\
\leq&\Delta t(1+\epsilon)^k\sum\limits_{q=1}^kJ^q
     +\zeta\Delta t^2s^{(k)}(1+\epsilon)^{k-1}\sum_{l=1}^{k-1}d_l^{(k)}\sum\limits_{q=1}^lJ^q\\
\leq&\Delta t(1+\epsilon)^k\sum\limits_{q=1}^kJ^q
     +\zeta\Delta t^2s^{(k)}(1+\epsilon)^k\sum\limits_{q=1}^{k-1}J^q\\
\leq&\Delta t(1+\epsilon)^k\big(1+\zeta\Delta ts^{(k)}\big)\sum\limits_{q=1}^kJ^q.
\end{align*}
Thus \eqref{error6} holds for $k=1,2,\ldots,n$ by mathematical induction. It remains to bound the right-hand side of \eqref{error6} for any $1\leq k\leq n$. We use Theorem \ref{truncation-error}, Lemma \ref{center-err} and \eqref{spatial-derivative} again to conclude that
\begin{align*}
\big\|r^n\big\|\leq\Delta t(1+\epsilon)^n\sum\limits_{k=1}^{n}J^k
&\leq c(1+\epsilon)^n\Delta t\sum\limits_{k=1}^{n}\frac 1{k^{\alpha^*}n^{1-\alpha^*}}+c(1+\epsilon)^n\Delta t\sum\limits_{k=1}^{n}\Delta x^2\\
&\leq c_4(1+\epsilon)^n\Delta t+c_5(1+\epsilon)^n\Delta x^2\\
&\leq c_6e^T\Delta t+c_7e^T\Delta x^2.
\end{align*}
We incorporate these estimates into \eqref{error6} to complete the proof.
\end{proof}

\section{Numerical results} \label{numerical-results}
In this section, we test some problems and present the numerical results to verify the effectiveness of the proposed  $RF$-$L1$ scheme  (\ref{RFS1})--(\ref{RFS3}) compared with $L1$ scheme (\ref{LS1})--(\ref{LS3}) and  $F$-$L1$ scheme (\ref{FS1})--(\ref{FS3}). All experiments are performed based on Matlab 2016b on a laptop with the configuration: Intel(R) Core(TM) i7-7500U CPU 2.70GHz and 8.00 GB RAM.

\begin{example}\label{example1}
To verify the efficiency of the robust fast algorithm for the VO Caputo fractional derivative, we first solve an ordinary differential equation
\begin{align*}
&\frac{\partial}{\partial t}u(t)+\zeta \prescript{C}{0}{\mathcal{D}}^{\alpha(t)}_tu(t)=1,\ \ t\in(0,T],\\
&u(0)=1,
\end{align*}
where $\zeta=1$, $T$=1 and the VO function is given by
\begin{align}\label{alpha}
\alpha(t)=\alpha(T)+\big(\alpha(0)-\alpha(T)\big)\bigg(1-t/T-\frac{\sin\big(2\pi(1-t/T)\big)}{2\pi}\bigg).
\end{align}
\end{example}
In the calculations, we use the numerical solutions $\widehat{U}$ to the corresponding problem discretized with $\Delta t=1/2^{22}$ as the reference solutions. Define the error and  the convergence rate in time by
\begin{eqnarray*}
E(\Delta t)=\left|u^n-\widehat{U}^n\right|,\ \ R_t=\log_2\frac{E(\Delta t)}{E(\Delta t/2)},
\end{eqnarray*}
respectively.

We set the expected accuracy $\epsilon=(\Delta t/T)^2$ to keep the accuracy of the solution of $RF$-$L1$ scheme as same as that of $L1$ scheme. The numerical results of $L1$ scheme, $F$-$L1$ scheme and $RF$-$L1$ scheme with different $\alpha(0)$ and $\alpha(T)$ are listed in Table \ref{T1}. For $\alpha_*=\alpha(0)=0.2$, Table \ref{T1} shows that compared with $L1$ scheme, the two fast algorithms $F$-$L1$ scheme and $RF$-$L1$ scheme greatly reduce the computational cost. The CPU time reveal $\mO(n\log^2 n)$ computational complexity of $F$-$L1$ scheme and $RF$-$L1$ scheme, and $\mO(n^2)$ computational complexity of $L1$ scheme, respectively. Although the computational complexity required in $F$-$L1$ scheme and $RF$-$L1$ scheme are both $\mO(n\log^2 n)$, the numbers of exponentials needed in $RF$-$L1$ scheme are very modest and not strongly influenced by $\alpha_*$, which indeed contributes to reduce the memory and computational cost. Moreover, as we have mentioned above, $F$-$L1$ scheme is not applicable for $\alpha_*$ of small value. Table \ref{T1} shows that $F$-$L1$ scheme cannot achieve the ideal convergence rate for $\alpha_*=0.05$ and even cannot work for $\alpha_*=0$. $RF$-$L1$ scheme always performs  well with few CPU time and memory.

\begin{example}\label{example2}
We investigate the temporal and spatial convergence behaviors of $RF$-$L1$ scheme. Consider \eqref{E1}--\eqref{E3} with the spatial domain $[x_l,x_r]=[0,1]$, the time interval $[0,T]=[0,1]$, $\zeta=1$, $p(x)=1$, $\varphi(x)=\sin(\pi x)$, $f=0$ and the VO function is given by \eqref{alpha}.
\end{example}
Set the expected accuracy $\epsilon=(\Delta t/T)^2$. We use the numerical solutions $\widehat{U}$ to the corresponding tFDE models discretized with $\Delta x=(x_r-x_l)/2^{10}$ and $\Delta t=T/2^{18}$ as the reference solutions. Define the error, the convergence rate in time and in space by
\begin{eqnarray*}
E(\Delta x,\Delta t)=\max_{0\leq j\leq m}\left|u_j^n-\widehat{U}_j^n\right|,\ \ R_t=\log_2\frac{E(\Delta x,\Delta t)}{E(\Delta x,\Delta t/2)},\ \ R_s=\log_2\frac{E(\Delta x,\Delta t)}{E(\Delta x/2,\Delta t)},
\end{eqnarray*}
respectively.

The error and temporal convergence order of $L1$ scheme, $F$-$L1$ scheme and $RF$-$L1$ scheme with different $\alpha(0)$ and $\alpha(T)$ are listed in Table \ref{T2}.  Fine spatial size is fixed at $\Delta x=(x_r-x_l)/2^{10}$. $F$-$L1$ scheme and $RF$-$L1$ scheme achieve the same accuracy as $L1$ scheme when $\alpha_*=\alpha(0)=0.2$. Compared with $L1$ scheme, $F$-$L1$ scheme greatly save the computational cost, and $RF$-$L1$ scheme further reduces CPU time and memory since much less $N_{\epsilon}$ is needed. Besides, $F$-$L1$ scheme  fails to solve the problem with $\alpha_*=0$ and is not very effective  for $\alpha_*=0.05$. Moreover, $RF$-$L1$ scheme is valid for $\alpha(0)=0$, $0.05$ and save much computational cost.

Table \ref{T3} lists the spatial convergence order of $L1$ scheme, $F$-$L1$ scheme and $RF$-$L1$ scheme with $\alpha(0)=0.05$ and $\alpha(T)=0.5$. Fine temporal step is fixed at $\Delta t=T/2^{18}$ and spatial sizes refined from $\Delta x=(x_r-x_l)/2^3$ to $\Delta x=(x_r-x_l)/2^7$. It shows that three schemes achieve the second-order convergence in space. Nevertheless, CPU time and memory of $RF$-$L1$ scheme are cheaper than those of $L1$ scheme and $F$-$L1$ scheme.

{\linespread{1.6}
\begin{table}[H]
\centering
\caption{Temporal convergence rates and the CPU time, memory of $L1$ scheme, $F$-$L1$ scheme and $RF$-$L1$ scheme for Example \ref{example1} }
\label{T1}
\resizebox{\textwidth}{!}{%
\begin{tabular}{cccccrccccrcrcccrcr}
\toprule
  & &&\multicolumn{4}{c}{$L1$ scheme}&&\multicolumn{5}{c}{$F$-$L1$ scheme}&&\multicolumn{5}{c}{$RF$-$L1$ scheme}\\
\cline{4-7}\cline{9-13}\cline{15-19}
$\big(\alpha(0),\alpha(T)\big)$  &$n$ &&$E(\Delta t)$&$R_t$&CPU(s)&Memory &&$E(\Delta t)$&$R_t$&CPU(s)&Memory&$\widetilde{N}_\epsilon$&&$E(\Delta t)$&$R_t$&CPU(s)&Memory&$N_\epsilon$\\
\midrule
                        &$2^{13}$   &&2.1281e-5 &-     &3.53     &1.97e+5  &&-         &-     &-     &-        &-     &&2.1281e-5 &-    &0.26  &4.90e+3 &98      \\
                        &$2^{14}$   &&1.0620e-5 &1.00  &12.69    &3.93e+5  &&-         &-     &-     &-        &-     &&1.0619e-5 &1.00 &0.50  &5.58e+3 &112     \\
$(0,0.2)$
                        &$2^{15}$   &&5.2890e-6 &1.00  &46.27    &7.87e+5  &&-         &-     &-     &-        &-     &&5.2889e-6 &1.01 &1.06  &6.30e+3 &127     \\
                        &$2^{16}$   &&2.6237e-6 &1.01  &179.23   &1.57e+6  &&-         &-     &-     &-        &-     &&2.6236e-6 &1.01 &2.20  &7.06e+3 &143     \\
                        &$2^{17}$   &&1.2910e-6 &1.02  &733.52   &3.15e+6  &&-         &-     &-     &-        &-     &&1.2910e-6 &1.02 &4.02  &7.83e+3 &159     \\
\hline
                        &$2^{13}$   &&1.9849e-5 &-     &3.29    &1.97e+5   &&1.4301e-5 &-     &0.97  &5.55e+4  &1153  &&1.9849e-5 &-    &0.30  &4.76e+5 &95      \\
                        &$2^{14}$   &&9.9041e-6 &1.00  &11.84   &3.93e+5   &&3.5916e-6 &1.99  &1.82  &6.40e+4  &1329  &&9.9040e-6 &1.00 &0.52  &5.48e+5 &110     \\
$(0.05,0.5)$
                        &$2^{15}$   &&4.9329e-6 &1.01  &47.42   &7.87e+5   &&1.4914e-6 &1.27  &3.09  &7.31e+4  &1519  &&4.9327e-6 &1.01 &1.20  &6.10e+5 &123     \\
                        &$2^{16}$   &&2.4476e-6 &1.01  &187.85  &1.57e+6   &&9.8072e-8 &3.93  &6.65  &8.27e+4  &1720  &&2.4473e-6 &1.01 &2.24  &6.87e+6 &139     \\
                        &$2^{17}$   &&1.2051e-6 &1.02  &767.61  &3.15e+6   &&2.0321e-7 &-1.05 &14.18 &9.31e+4  &1935  &&1.2049e-6 &1.02 &4.58  &7.69e+6 &156     \\
\hline
                        &$2^{13}$   &&1.8761e-5 &-     &3.19    &1.97e+7   &&1.8754e-5 &-     &0.65  &1.54e+4  &317   &&1.8761e-5 &-    &0.50  &4.52e+5 &90      \\
                        &$2^{14}$   &&9.3607e-6 &1.00  &12.50   &3.93e+7   &&9.3583e-6 &1.00  &0.95  &1.77e+4  &365   &&9.3605e-6 &1.00 &0.66  &5.10e+5 &102     \\
$(0.2,0.6)$
                        &$2^{15}$   &&4.6624e-6 &1.01  &48.05   &7.87e+7   &&4.6601e-6 &1.01  &2.01  &2.02e+4  &416   &&4.6622e-6 &1.01 &1.31  &5.77e+5 &116     \\
                        &$2^{16}$   &&2.3139e-6 &1.01  &188.77  &1.57e+8   &&2.3137e-6 &1.01  &4.06  &2.28e+4  &471   &&2.3135e-6 &1.01 &2.65  &6.44e+5 &130     \\
                        &$2^{17}$   &&1.1399e-6 &1.02  &759.75  &3.15e+8   &&1.1398e-6 &1.02  &8.20  &2.56e+4  &529   &&1.1397e-6 &1.02 &4.59  &7.11e+6 &144     \\
\bottomrule
\end{tabular}}
\end{table}
}

{\linespread{1.6}
\begin{table}[H]
\centering
\caption{Temporal convergence rates and the CPU time, memory of $L1$ scheme, $F$-$L1$ scheme and $RF$-$L1$ scheme for Example \ref{example2} with $m=2^{10}$}
\label{T2}
\resizebox{\textwidth}{!}{%
\begin{tabular}{cccccrccccrcrcccrcr}
\toprule
  & &&\multicolumn{4}{c}{$L1$ scheme}&&\multicolumn{5}{c}{$F$-$L1$ scheme}&&\multicolumn{5}{c}{$RF$-$L1$ scheme}\\
\cline{4-7}\cline{9-13}\cline{15-19}
$\big(\alpha(0),\alpha(T)\big)$  &$n$ &&$E(\Delta x,\Delta t)$&$R_t$&CPU(s)&Memory &&$E(\Delta x,\Delta t)$&$R_t$&CPU(s)&Memory&$\widetilde{N}_\epsilon$&&$E(\Delta x,\Delta t)$&$R_t$&CPU(s)&Memory&$N_\epsilon$\\
\midrule
                        &$2^{11}$   &&6.5685e-6 &-     &17.34    &1.69e+7   &&-         &-    &-       &-        &-      &&6.5685e-6 &-    &1.41   &7.48e+5 &73      \\
                        &$2^{12}$   &&3.2568e-6 &1.01  &66.37    &3.37e+7   &&-         &-    &-       &-        &-      &&3.2568e-6 &1.01 &3.06   &8.46e+5 &85     \\
$(0,0.2)$
                        &$2^{13}$   &&1.6022e-6 &1.02  &284.29   &6.73e+7   &&-         &-    &-       &-        &-      &&1.6022e-6 &1.02 &6.26   &9.53e+5 &98     \\
                        &$2^{14}$   &&7.7515e-7 &1.05  &1170.43  &1.34e+8   &&-         &-    &-       &-        &-      &&7.7515e-7 &1.05 &14.14  &1.07e+6 &112     \\
                        &$2^{15}$   &&3.6171e-7 &1.10  &4599.13  &2.69e+8   &&-         &-    &-       &-        &-      &&3.6171e-7 &1.10 &35.79  &1.19e+6 &127     \\
\hline
                        &$2^{11}$   &&1.4465e-5 &-     &17.53    &1.69e+7   &&1.2593e-5 &-    &43.91   &7.03e+6  &837    &&1.4465e-5 &-    &1.11   &7.31e+5 &71      \\
                        &$2^{12}$   &&7.1687e-6 &1.01  &75.82    &3.37e+7   &&6.1145e-6 &1.04 &99.49   &8.28e+6  &989    &&7.1687e-6 &1.01 &2.22   &8.30e+5 &83     \\
$(0.05,0.5)$
                        &$2^{13}$   &&3.5253e-6 &1.02  &278.57   &6.73e+7   &&3.2392e-6 &0.92 &232.70  &9.63e+6  &1153   &&3.5253e-6 &1.02 &5.86   &9.29e+5 &95     \\
                        &$2^{14}$   &&1.7051e-6 &1.05  &1098.39  &1.34e+8   &&1.3753e-6 &1.24 &530.39  &1.11e+7  &1329   &&1.7051e-6 &1.05 &15.34  &1.05e+6 &110     \\
                        &$2^{15}$   &&7.9551e-7 &1.10  &4429.14  &2.69e+8   &&6.1283e-7 &1.17 &1211.36 &1.26e+7  &1519   &&7.9551e-7 &1.10 &34.56  &1.16e+6 &123     \\
\hline
                        &$2^{11}$   &&1.6780e-5 &-     &17.20    &1.69e+7   &&1.6777e-5 &-    &12.49   &2.05e+6  &231    &&1.6780e-5 &-    &0.99   &6.98e+5 &67      \\
                        &$2^{12}$   &&8.3079e-6 &1.01  &67.22    &3.37e+7   &&8.3049e-6 &1.01 &30.18   &2.38e+6  &272    &&8.3078e-6 &1.01 &1.99   &7.89e+5 &78     \\
$(0.2,0.6)$
                        &$2^{13}$   &&4.0826e-6 &1.02  &268.00   &6.73e+7   &&4.0822e-6 &1.02 &66.48   &2.75e+6  &317    &&4.0826e-6 &1.02 &4.30   &8.88e+5 &90     \\
                        &$2^{14}$   &&1.9736e-6 &1.05  &1073.83  &1.34e+8   &&1.9735e-6 &1.05 &152.53  &3.15e+6  &365    &&1.9736e-6 &1.05 &9.58   &9.86e+5 &102     \\
                        &$2^{15}$   &&9.2040e-7 &1.10  &4565.72  &2.69e+8   &&9.2026e-7 &1.10 &343.53  &3.57e+6  &416    &&9.2040e-7 &1.10 &23.17  &1.10e+6 &116     \\
\bottomrule
\end{tabular}}
\end{table}
}

{\linespread{1.6}
\begin{table}[H]
\centering
\caption{Spatial convergence rates and the CPU time, memory of $L1$ scheme, $F$-$L1$ scheme and $RF$-$L1$ scheme for Example \ref{example2} with $n=2^{18}$}
\label{T3}
\resizebox{\textwidth}{!}{%
\begin{tabular}{cccccrccccrcrcccrcr}
\toprule
  & &&\multicolumn{4}{c}{$L1$ scheme}&&\multicolumn{5}{c}{$F$-$L1$ scheme}&&\multicolumn{5}{c}{$RF$-$L1$ scheme}\\
\cline{4-7}\cline{9-13}\cline{15-19}
$\big(\alpha(0),\alpha(T)\big)$  &$m$ &&$E(\Delta x,\Delta t)$&$R_s$&CPU(s)&Memory &&$E(\Delta x,\Delta t)$&$R_s$&CPU(s)&Memory&$\widetilde{N}_\epsilon$&&$E(\Delta x,\Delta t)$&$R_s$&CPU(s)&Memory&$N_\epsilon$\\
\midrule
                        &$2^3$   &&9.2958e-4 &-     &5533.98   &1.89e+7   &&9.2955e-4 &-    &82.48   &2.33e+4  &2161    &&9.2958e-4 &-    &15.90   &4.20e+4 &172       \\
                        &$2^4$   &&2.3079e-4 &2.01  &7586.97   &3.57e+7   &&2.3076e-4 &2.01 &102.65  &3.72e+5  &2161    &&2.3079e-4 &2.01 &18.05   &5.40e+4 &172     \\
$(0.05,0.5)$
                        &$2^5$   &&5.7557e-5 &2.00  &11751.81  &6.92e+7   &&5.7527e-5 &2.00 &150.06  &6.51e+5  &2161    &&5.7557e-5 &2.00 &26.15   &7.79e+4 &172     \\
                        &$2^6$   &&1.4341e-5 &2.00  &20286.87  &1.36e+8   &&1.4311e-5 &2.01 &968.32  &1.21e+6  &2161    &&1.4341e-5 &2.00 &49.96   &1.26e+5 &172     \\
                        &$2^7$   &&3.5427e-6 &2.02  &37741.75  &2.71e+8   &&3.5132e-6 &2.03 &2204.64 &2.32e+6  &2161    &&3.5427e-6 &2.02 &63.77   &2.22e+5 &172     \\

\bottomrule
\end{tabular}}
\end{table}
}

\section{Concluding Remarks} \label{concluding-remarks}
In this paper, a robust fast algorithm is developed to approximate the VO Caputo fractional derivative, which can handle the cases of small or vanishing lower bound of the VO function. The method is applied to construct a fast finite difference scheme for the VO tFDEs. Moreover, the convergence is studied without any regularity assumptions of the true solution. Numerical tests are reported to show the effectiveness of the proposed scheme and confirm the theoretical findings.

In our future work, fast methods for high-order scheme to approximate the VO tFDEs with the initial weak singularity are meaningful to study.

\bibliographystyle{unsrt}

\begin{thebibliography}{10}

\bibitem{A-2003}
{R.A. Adams and J.J.F. Fournier}, Sobolev spaces, Elsevier, San Diego, 2003.


\bibitem{Baffet-2016}
D. Baffet and J.S. Hesthaven, A kernel compression scheme for fractional differential equations, SIAM J. Numer. Anal., 55 (2017), pp. 496--520.

\bibitem{Benson-2000}
D.A. Benson, S.W. Wheatcraft, and M.M. Meerschaert, Application of a fractional advection dispersion equation, Water Resour. Res., 36 (2000), pp. 1403--1412.

\bibitem{Bertaccini-2019}
{D. Bertaccini and F. Durastante}, Block structured preconditioners in tensor form for the all-at-once solution of a finite volume fractional diffusion equation, Appl. Math. Lett., 95 (2019), pp. 92--97.


\bibitem{Beylkin-2017}
G. Beylkin and L. Monzon, Approximation by exponential sums revisited, Appl. Comput. Harmon. Anal.,
28 (2010), pp. 131--149.

\bibitem{C}
{A.V. Chechkin, R. Gorenflo, and I.M. Sokolov}, Fractional diffusion in inhomogeneous media, J. Phys. A: Math. Gen., 38 (2005), pp. 679--684.


\bibitem{Coimbra-2003}
C.F.M. Coimbra, Mechanics with variable-order differential operators, Ann. Phys., 12 (2003), pp. 643--736.


\bibitem{Diazand-2009}
G. Diazand and C.F.M. Coimbra, Nonlinear dynamics and control of a variable order oscillator with application to the van der pol equation, Nonlinear Dynam., 56 (2009), pp. 145--157.

 \bibitem{Du-2020}
{R.L. Du, A.A. Alikhanov, and Z.Z. Sun}, Temporal second order difference schemes for the multi-dimensional variable-order time fractional sub-diffusion equations, Comput. Math. with Appl., 79 (2020), pp. 2952--2972.

\bibitem{Fang-2019}
Z.W. Fang, H.W. Sun, and H. Wang, A fast method for variable-order caputo fractional derivative with applications to time fractional diffusion equations, Comput. Math. Appl., 80 (2020), pp. 1443--1458.



\bibitem{Fu-2019}
Z.J. Fu, S. Reutskiy, H.G. Sun, J. Ma, and M.A. Khan, A robust kernel-based solver for variable-order time fractional PDEs under 2D/3D irregular domains, Appl. Math. Lett., 94 (2019), pp. 105--111.

\bibitem{Fu-2017}
{H.F. Fu and H. Wang}, A preconditioned fast finite difference method for space-time fractional partial differential equations, Fract. Calc. Appl. Anal., 20 (2017), pp. 88--116.

%

\bibitem{Gu-2021}
X.M. Gu, H.W. Sun, Y. L. Zhao, and X. C. Zheng, An implicit difference scheme for time-fractional diffusion equations with a time-invariant type variable order, Appl. Math. Lett., DOI: https://doi.org/10.1016/j.aml.2021.107270, 2021.


\bibitem{Ingman-2004}
{D. Ingman and J. Suzdalnitsky}, Control of damping oscilations by fractional differential operator with time-dependent order, Comput. Methods Appl. Mech. Eng., 193 (2004), pp. 5585--5595.


\bibitem{Jia-2017}
Y. Jia, M. Xu, and Y. Lin, A numerical solution for variable order fractional functional differential equation, Appl. Math. Lett., 64 (2017), pp. 125--130.

\bibitem{Jiang-2017}
S.D. Jiang, J.W. Zhang, Q. Zhang, and Z.M. Zhang, Fast evaluation of the caputpo fractional derivative and its applications to fractional diffusion equations, Commun. Comput. Phys.,
21 (2017), pp. 650--678.


\bibitem{Ke-2015}
R.H. Ke, M.K. Ng, and H.W. Sun, A fast direct method for block triangular Toeplitz-like with tridiagonal block systems from time-fractional partial differential equations, J. Comput. Phys., 303 (2015), pp. 203--211.

\bibitem{Kilbas-2006}
{A.A. Kilbas, H.M. Srivastava, and J.J. Trujillo}, Theory and applications of fractional differential equations, Elsevier, Amsterdam, 2006.



\bibitem{Langlands-2005}
{T.A. M. Langlands and B.I. Henry}, The accuracy and stability of an implicit solution method for the fractional diffusion equation, J. Comput. Phys., 205 (2005), pp. 719--736.

\bibitem{Lorenzo-2002}
C.F. Lorenzo and T.T. Hartley, Variable order and distributed order fractional operators, Nonlin. Dyn., 29 (2002), pp. 57--98.

\bibitem{Lu-2015}
 X. Lu, H.W. Sun, and H.K. Pang, Fast approximate inversion of a block triangular Toeplitz matrix
with applications to fractional subdiffusion equations, J. Comput. Phys., 22 (2015), pp. 866--882.

\bibitem{Liao-2018}
{H.L. Liao, D.F. Li, and J.W. Zhang}, Sharp error estimate of the nonuniform $L1$ formula for linear reaction-subdiffusion equations, SIAM J. Numer. Aanl., 56 (2018), pp. 1112--1133.


\bibitem{Liu-2004}
{F.W. Liu, V. Anh, and I. Turner}, Numerical solution of the space fractional Fokker-Planck equation, J. Comput. Appl. Math., 166 (2004), pp. 209--219.

\bibitem{Lubich-2002}
{C. Lubich and A. Sch\"{a}dle}, Fast convolution for nonreflecting boundary conditions, SIAM J. Sci. Comput., 24 (2002), pp. 161--182.


\bibitem{Mainardi-2000}
{F. Mainardi, M. Raberto, R. Gorenflo, and E. Scalas}, Fractional calculus and continuous-time finance II: the waiting-time distribution, Phys. A, 287 (2000), pp. 468--481.

\bibitem{Oldham-1974}
K.B. Oldham and J. Spanier, The fractional calculus, Academic Press, New York, 1974.

\bibitem{Obembe-2017}
A.D. Obembe, M.E. Hossain, and S.A. Abu-Khamsin, Variable-order derivative time
fractional diffusion model for heterogeneous porous media, J. Petrol. Sci. Eng., 152 (2017), pp. 391--405.



\bibitem{Patnaik-2020}
{S. Patnaik, J.P. Hollkamp, and F. Semperlotti}, Applications of variable-order fractional operators: a review, Proc. R. Soc. Lond. Ser. A. Math. Phys. Eng. Sci., DOI: http://dx.doi.org/10.1098/rspa.2019.0498, 2020.

\bibitem{Pedro-2008}
{H.T. C. Pedro, M.H. Kobayashi, J.M. C. Pereira, and C.F. M. Coimbra}, Variable order modeling of diffusive-convective effects on the oscillatory flow past a sphere, J. Vib. Control, 14 (2008), pp. 1659--1672.

\bibitem{Podlubny-1999}
{I. Podlubny}, Fractional differential equations, Academic Press, New York, 1999.


\bibitem{Raberto-2002}
{M. Raberto, E. Scalas, and F. Mainardi}, Waiting-times and returns in high-frequency financial data: an empirical study, Phys. A, 314 (2002), pp. 749--755.


\bibitem{Sakamoto-2011}
{K. Sakamoto and M. Yamamoto}, Initial value/boundary value problems for fractional diffusion-wave equations and applications to some inverse problems, J. Math. Anal. Appl., 382 (2011), pp. 426--447.

\bibitem{Sokolov-2005}
{I.M. Sokolov and J. Klafter}, From diffusion to anomalous diffusion: a century after einsteins brownian motion, Chaos, 15 (2005), pp. 1--7.



\bibitem{Sun-2005}
Z.Z. Sun, Numerical methods of partial differential equations, Science Press, Beijing, 2005.


\bibitem{Sun-2009}
H.G. Sun, W. Chen, and Y.Q. Chen, Variable-order fractional differential operators in anomalous diffusion modeling, Phys. A, 388 (2009), pp. 4586--4592.

\bibitem{Sun-2019}
{H.G. Sun, A. Chang, Y. Zhang, and W. Chen}, A review on variable-order fractional differential equations: mathematical foundations, physical models, numerical methods and applications, Fract. Calc. Appl. Anal., 22 (2019), pp. 27--59.

\bibitem{Sun-2011}
H.G. Sun, W. Chen, H. Wei, and Y.Q. Chen, A comparative study of constant-order and variable-order fractional models in characterizing memory property of systems, Eur. Phys. J. Spec. Top. 193 (2011), pp. 185--192.

\bibitem{Sun-2012}
{H.G. Sun, W. Chen, C.P. Li, and Y.Q. Chen}, Finite difference schemes for variable-order time fractional diffusion equation, Int. J. Bifurcation Chaos, 22  (2012), pp. 1250085, 16 pages.

\bibitem{T-1984}
V. Thom$\acute{e}$e, Galerkin Finite Element Methods for Parabolic Problems, Springer, New York, 1984.

\bibitem{T-2013}
J.A. Trangenstein, Numerical solution of elliptic and parabolic partial differential equations, Cambridge University Press, New York, 2013.


\bibitem{Zeng-2017}
{F.H. Zeng, I. Turner, and K. Burrage}, A stable fast time-stepping method for fractional integral and derivative operators, J. Sci. Comput., 77 (2018), pp. 283--307.

\bibitem{Zhang-2021}
{J.L. Zhang, Z.W. Fang, and H.W. Sun}, Exponential-sum-approximation technique for variable-order time-fractional diffusion equations, J. Appl. Math. Comput., DOI: https://doi.org/10.1007/s12190-021-01528-7, 2021.

\bibitem{Zhao-2015}
{X. Zhao, Z.Z. Sun, and G.E. Karniadakis}, Second-order approximations for variable order fractional derivatives: Algorithms and applications, J. Comput. Phys., 293 (2015), pp. 184--200.

\bibitem{Zheng-2020}
{X.C. Zheng and H. Wang}, Optimal-order error estimates of finite element approximations to variable-order time-fractional diffusion equations without regularity assumptions of the true solutions, IMA J. Numer. Anal., 41 (2021), pp. 1522--1545.

\bibitem{Zheng-2019}
{H. Wang and X.C. Zheng}, Wellposedness and regularity of the variable-order time-fractional diffusion equations, J. Math. Anal. Appl., 475 (2019), pp. 1778--1802.

\bibitem{Zhuang-2000}
{P. Zhuang, F.W. Liu, V. Anh, and I. Turner}, Numerical methods for the variable-order fractional advection-diffusion equation with a nonlinear source term, SIAM J. Numer. Anal., 47 (2009), pp. 1760--1781.



\end{thebibliography}

\end{document}